\documentclass[fleqn,reqno,12pt]{amsart}
\usepackage{graphicx}
\usepackage{amsmath}
\usepackage{amsfonts}

\newtheorem{theorem}{Theorem}
\theoremstyle{plain}

\newtheorem{definition}{Definition}

\newtheorem{lemma}{Lemma}
\newtheorem{notation}{Notation}

\newtheorem{remark}{Remark}

\numberwithin{equation}{section}

\begin{document}
\title[homogenization and bioheat equation]{ON A MULTISCALE ANALYSIS OF A
MICRO-MODEL OF HEAT TRANSFER IN BIOLOGICAL TISSUES}
\author{Abdelhamid AINOUZ}
\address{Laboratory AMNEDP, Faculty of Mathematics, University of sciences
and technology, Po Box 32 El Alia, Algiers, Algeria.}
\email{aainouz@usthb.dz}
\thanks{The paper is a part of the project N$%
{{}^\circ}%
$ B00002278 of the MESRS algerian office. The author is indebted for their
financial support.}
\date{}

\begin{abstract}
A bio-heat transfer model for biological tissues in a micro-scale and
periodical settings is investigated . It is assumed that the model is a
two-component system consisting of solid particles representing tissue cells
and interconnected pores containing either arterial or venous blood. This
tissue-blood system is described by two energy equations, one equation for
the solid tissues and the other for the surrounding blood. On the interface
between them, it is assumed that the heat transfer is governed by Newton's
cooling law. Using homogenization techniques, it is shown that the obtained
macro-model presents some extra-terms and it can be seen as a new
mathematical model for human thermotherapy and human thermoregulation system.
\end{abstract}

\maketitle

\section{Introduction\label{s1}}

Studying heat transfer in biological tissues is important  in many
biomedical engineering such as thermoregulation system, thermotherapy and
radiotherapy, skin surgery etc... See for instance S.A. Berger \& \textit{al.%
} \cite{bgl}, J.C. Chato \cite{cha}, M. Gautherie \cite{gau}, K. Khanafer \&
\textit{al. }\cite{kapb}, M. Miyikawa and J.C. Bolomey (eds) \cite{mik} and
the references therein. Many mathematical models were proposed to predict
the distribution of the temperaure in biological tissues. One of the most
widely used model is the bioheat equation after the pioneering work of H.H.
Pennes \cite{pen}. It is based on the well-known Fourier law with the
concept of blood perfusion. It reads as follows
\begin{equation}
\rho c\partial _{t}T-\text{\textrm{div}}\left( \kappa \nabla T\right)
+\omega _{b}\rho _{b}c_{b}\rho \left( T-T_{a}\right) =f  \label{bm1}
\end{equation}%
where $T,\rho ,c$ and $\kappa $ are respectively the temperature, the
density, the specific heat and the heat conductivity coefficient of the
tissue, $\omega _{b}$ is the blood perfusion, $\rho _{b}$ and $c_{b}$ are
the density and the specific heat of the blood and $T_{a}$ is the
temperature of the arterial flow. Finally $f$ is some external source of
heating and it is generally written as the sum of sources due to absorbed
laser light and metabolic activity. Many scientists have attempted to
justify the Helmoltz term of (\ref{bm1}). We mention for example M.M. Chen,
K.R. Holmes\cite{ch}, P. Wust \& \textit{al.} \cite{wnfp}, R. Hochmuth and
P. Deuflhard\cite{hd}. In the latter, a homogenization technique was
developped on a microvascular model consisting of tissues (solid) surrounded
by blood (fluid). More precisely they study the following system%
\begin{eqnarray}
-\Delta T^{\varepsilon } &=&S^{\varepsilon }\text{ in }\Omega ^{\varepsilon }%
\text{,}  \label{bm2} \\
T^{\varepsilon } &=&0\text{ on }\partial \Omega ,  \label{bm3} \\
\frac{\partial T^{\varepsilon }}{\partial n^{\varepsilon }} &=&\varepsilon
\alpha \left( T_{b}^{\varepsilon }-T^{\varepsilon }\right) \text{ on }%
\partial Q^{\varepsilon }  \label{bm4}
\end{eqnarray}%
where $T^{\varepsilon }$ is the temperature and $S^{\varepsilon }$ the
source term in the tissues represented by $\Omega ^{\varepsilon }$ whereas $%
T_{b}^{\varepsilon }$ is the temperature of the blood. The domain $\Omega
^{\varepsilon }$ is obtained by removing from $\Omega $, a bounded and
regular domain, a set of holes $Q^{\varepsilon }$ where blood flows. In (\ref%
{bm4}), $n^{\varepsilon }$ is the unit normal of $\partial Q^{\varepsilon }$
outward to $\Omega ^{\varepsilon }$ and $\alpha >0$ a physiological
parameter. They obtained the following homogenized model:
\begin{equation*}
-\text{\textrm{div}}\left( \mathcal{A}\nabla T^{\ast }\right) +\alpha ^{\ast
}\left( T^{\ast }-T_{b}\right) =S
\end{equation*}%
where $\mathcal{A}$ is the homogenized tensor, $T^{\ast }$ is the weak limit
in $H^{1}\left( \Omega \right) $ of some extended temperature $%
P^{\varepsilon }\left( T^{\varepsilon }\right) $, $S^{\varepsilon }$ is the
weak limit in $L^{2}\left( \Omega \right) $ of the source term $%
S^{\varepsilon }$ and finally $\alpha ^{\ast }$ is the effective Helmoltz
term. In fact, biological tissues can be seen as porous media where cells
(matrix) are separated by voids or pores which are filled with blood. This
system of cells-blood can be interpreted as a two-constituent medium. In
connection with binary composites presenting thermal barriers at the
interfacial contact, we mention especially the work by J.L.\ Auriault and H.
Ene\cite{ae} where they study heat transfer in a two-component composite
with conductivities of the same order of magnitude. The macroscopic model is
shown to belong to two main types of field models: one-temperature and
two-temperature, depending on the order of magnitude of the interfacial
thermal conductance. In the present paper, we shall be concerned with a
micro-model for the heat transfer in a biological tissue\ made of two
interacting systems ( cells tissues and blood regions) where the
conductivities are assumed to be of different order of magnitude. We also
assume that the transition between these two regions on the interface is
governed by Newton's cooling law. That is the heat flow through the
interface is proportional, by the thermal conductance of the layer, to the
jump of the temperature field, see R. Hochmuth and P. Deuflhard\cite{hd}
(see also H.S. Carslaw and J.C. Jaeger\cite{cj}). We mention that this kind
of boundary transmission condition was used for the homogenization in porous
media, see for e.g. A. Ainouz\cite{ain}$^{,}$\cite{ain1}$^{,}$\cite{ain2}.
In fact, there are many works showing how transport theories in porous media
enhance the understanding of flow and heat transfer in biological tissues.
For more details, we refer the reader to the survey paper by A.-R.A. Khaled
and K. Vafai\cite{kv}.

The paper is organized as follows: in Section 2, the geometry of the domain
and the micro-model are set. In Section 3 \ a formal expansion technique is
used to derive the homogenized model. Finally in Section 4, the two-scale
convergence technique is applied to justify the formal procedure of Section
3.

\section{Setting of the Problem\label{s2}}

We start by introducing the notation used throughout this paper. We consider
$\Omega $\ a bounded domain in $\mathbb{R}^{d}\ $($d\geq 2$) with smooth
boundary $\partial \Omega $. Let $Y\overset{\mathrm{def}}{=}\ ]0,1[^{d}$ be
the generic cell of periodicity divided as $Y=Y_{1}\cup Y_{2}\cup \Sigma $
where $Y_{1},$\ $Y_{2}$\ are two connected, open disjoint subsets of $Y$ and
$\Sigma \overset{\mathrm{def}}{=}\ \partial Y_{1}\cap \partial Y_{2}$\ is a
smooth $\left( d-1\right) $-dimensional manifold. As in G. Allaire and F.
Murat\cite{am}, we assume that the $Y-$periodic continuation of $Y_{1}$\ to
the whole space $\mathbb{R}^{d}$, namely $\widetilde{Y_{1}}=\cup _{k\in
\mathbb{Z}^{d}}\left( k+Y_{1}\right) $\ is smooth and connected. Note that
no connectedness assumption is made on the part $\cup _{k\in \mathbb{Z}%
^{d}}\left( k+Y_{2}\right) $.

Let $\chi _{1}$ (resp. $\chi _{2}$) denote the $Y$-periodic characteristic
function of $Y_{1}$ (resp. $Y_{2}$). Denoting $\varepsilon >0$ a
sufficiently small parameter, we set%
\begin{equation*}
\Omega _{1}^{\varepsilon }\ \overset{\mathrm{def}}{=}\ \{x\in \Omega :\chi
_{1}(\frac{x}{\varepsilon })=1\},\ \Omega _{2}^{\varepsilon }\ \overset{%
\mathrm{def}}{=}\ \{x\in \Omega :\chi _{2}(\frac{x}{\varepsilon })=1\},
\end{equation*}%
and let $\Sigma ^{\varepsilon }\ \overset{\mathrm{def}}{=}\ \overline{\Omega
_{1}^{\varepsilon }}\cap \overline{\Omega _{2}^{\varepsilon }}$. Without
loss of generality, we assume that the region $\Omega _{2}^{\varepsilon }$\
is strictly embedded in the region $\Omega _{1}^{\varepsilon }$, in the
sense that $\overline{\Omega _{2}^{\varepsilon }}\subset \Omega $. In this
connection, $\Omega _{1}^{\varepsilon }$ is referred as the cellular domain
and $\Omega _{2}^{\varepsilon }$ as the voids filled with blood. We see that
the boundary of $\Omega _{2}^{\varepsilon }$ is the interface $\Sigma
^{\varepsilon }$ and the boundary of $\Omega _{1}^{\varepsilon }$ consists
then of two parts: $\Sigma ^{\varepsilon }$ and the exterior boundary $%
\Gamma $. We can write that
\begin{equation*}
\partial \Omega _{2}^{\varepsilon }=\Sigma ^{\varepsilon }\text{ and }%
\partial \Omega _{1}^{\varepsilon }=\partial \Omega \cup \Sigma
^{\varepsilon }\text{.}
\end{equation*}%
Thanks to the connectedness of $\widetilde{Y_{1}}$, we see that $\Omega
_{1}^{\varepsilon }$ is connected while $\Omega _{2}^{\varepsilon }$ may or
may not be connected.

Let us denote $\rho ,c$ and $\kappa $ the density, the specific heat and the
heat conductivity coefficient of the tissue, respectively. Let $\omega _{b}$
denote the blood perfusion, $\rho _{b}$ and $c_{b}$ the density and the
specific heat of the blood. Let $\kappa _{b}$ the heat conductivity
coefficient of the blood. We shall assume that the phenomenological
parameters: $\rho ,c,\kappa ,\omega _{b},\rho _{b},c_{b}$ and $\kappa _{b}$
are positive constant and independent of $\varepsilon $.

Let $\left( 0,T\right) $ be the time interval. Put
\begin{eqnarray*}
&&Q\ \overset{\mathrm{def}}{=}\ \left( 0,T\right) \times \Omega ,\ \ \Gamma
\ \overset{\mathrm{def}}{=}\left( 0,T\right) \times \partial \Omega ,\ \ S\
\overset{\mathrm{def}}{=}\ \left( 0,T\right) \times \Sigma ,\ \  \\
&&Q_{1}^{\varepsilon }\overset{\mathrm{def}}{=}\ \left( 0,T\right) \times
\Omega _{1}^{\varepsilon },\ Q_{2}^{\varepsilon }\overset{\mathrm{def}}{=}\
\left( 0,T\right) \times \Omega _{2}^{\varepsilon },\ S^{\varepsilon }\
\overset{\mathrm{def}}{=}\ \left( 0,T\right) \times \Sigma ^{\varepsilon }
\end{eqnarray*}%
Let also
\begin{equation}
\alpha =\frac{\kappa }{\rho c},\ \alpha _{b}=\frac{\kappa }{\rho _{b}c_{b}}%
,\ \alpha _{b}^{\varepsilon }=\varepsilon ^{2}\alpha _{b},\ \gamma =\omega
_{b}\rho _{b}\frac{c_{b}}{c}.  \label{a1}
\end{equation}

The micro-model that we shall study in this paper is as follows:
\begin{subequations}
\begin{gather}
\partial _{t}T^{\varepsilon }-\alpha \Delta T^{\varepsilon }=F\text{\ \ in\ }%
Q_{1}^{\varepsilon },  \label{dp4p1} \\
\partial _{t}T_{b}^{\varepsilon }-\alpha _{b}^{\varepsilon }\Delta
T_{b}^{\varepsilon }=F_{b}\text{\ \ \ in }Q_{2}^{\varepsilon },
\label{dp4p2} \\
\alpha \nabla T^{\varepsilon }\cdot \nu ^{\varepsilon }=\alpha
_{b}^{\varepsilon }\nabla T_{b}^{\varepsilon }\cdot \nu ^{\varepsilon }\text{%
\ \ \ on }S^{\varepsilon },  \label{dp4p3} \\
\alpha \nabla T^{\varepsilon }\cdot \nu ^{\varepsilon }=-\varepsilon \gamma
\left( T^{\varepsilon }-T_{b}^{\varepsilon }\right) \text{\ \ \ on }%
S^{\varepsilon },  \label{dp4p4} \\
T^{\varepsilon }=0\text{\ \ \ on }\Gamma ,  \label{dp4p5} \\
T^{\varepsilon }\left( 0,\cdot \right) =h\left( \cdot \right) \text{\ \ \ in
}\Omega _{1}^{\varepsilon },  \label{dp4p6} \\
T_{b}^{\varepsilon }\left( 0,\cdot \right) =h_{b}\left( \cdot \right) \text{%
\ \ \ in }\Omega _{2}^{\varepsilon }  \label{dp4p7}
\end{gather}%
where $f$ (resp. $f_{b}$) be some external source of heating in cells (resp.
blood), $\nu ^{\varepsilon }$ stands for the unit normal of $\Sigma
^{\varepsilon }$ outward to $\Omega _{1}^{\varepsilon }$ and $h,h_{b}$ are
the initial temperature field in $\Omega _{1}^{\varepsilon }$, $\Omega
_{2}^{\varepsilon }$ respectively. Without no loss of generality, we shall
assume
\end{subequations}
\begin{equation}
f,\ f_{b},\ h,\ h_{b}\in L^{2}\left( \Omega \right) .  \label{dp4h7}
\end{equation}%
Our system actually models heat flow in a porous medium (biological tissue).
In this connection, $\Omega _{1}^{\varepsilon }$ represents the matrix-cells
space region and $\Omega _{2}^{\varepsilon }$ the pores which are filled
with blood. The thin layer $\Sigma ^{\varepsilon }$ is an interfacial flow
barrier with heat conductance given by $\gamma ^{\varepsilon }=\varepsilon
\omega _{b}\rho _{b}c_{b}\rho $. The unknowns $T^{\varepsilon }$ and $%
T_{b}^{\varepsilon }$ are the temperatures in $Q_{1}^{\varepsilon }$ and $%
Q_{2}^{\varepsilon }$ respectively. The first equation describes the heat
flow in the cells with large conductivity and the second describes the heat
flow in the blood region with low conductivity. Condition (\ref{dp4p3})
expresses flux continuity across the interface. However, the temperature may
present in general jumps across $\Sigma ^{\varepsilon }$. Here, we have
employed Newton's cooling law described by (\ref{dp4p4}), see for instance
R. Hochmuth and P. Deuflhard\cite{hd} and H.S. Carslaw and J.C. Jaeger\cite%
{cj}. We can say that $\Omega _{1}^{\varepsilon }$ can be considered as a
good conductor, while $\Omega _{2}^{\varepsilon }$ a poor one. The interface
$\Sigma ^{\varepsilon }$ can be seen as a heat exchanger. The condition (\ref%
{dp4p5}) is the standard homogeneous Dirichlet condition on the exterior
boundary. Finally equations (\ref{dp4p6})-(\ref{dp4p7}) are the initial
conditions, which close the system under consideration. Note that in (\ref%
{a1}), the matrix $A_{b}^{\varepsilon }$ is scaled by $\varepsilon ^{2}$ to
provide the correct scaling for the heat flow in the block regions. Indeed,
this scaling is the unique choice that makes every term of the porous medium
equation in the block cell reappears in the leading order asymptotic
expansion, so that the form of the equation is preserved on the small scale
independently of\textbf{\ }$\varepsilon $.

Let us now establish a variational framework of our problem. To this end, we
first introduce some notations. Let
\begin{equation*}
H^{\varepsilon }=L^{2}\left( \Omega _{1}^{\varepsilon }\right) \times
L^{2}\left( \Omega _{2}^{\varepsilon }\right) \text{ and}\ V^{\varepsilon
}=\left( H^{1}\left( \Omega _{1}^{\varepsilon }\right) \cap H_{0}^{1}\left(
\Omega \right) \right) \times H^{1}\left( \Omega _{2}^{\varepsilon }\right) .
\end{equation*}%
We shall consider on $H^{\varepsilon }$ and $V^{\varepsilon }$ the following
inner products:
\begin{eqnarray*}
\left( \varphi ,\psi \right) _{H^{\varepsilon }} &=&\int_{\Omega
_{1}^{\varepsilon }}\varphi _{1}\psi _{1}\hspace{0.01in}\mathrm{d}%
x+\int_{\Omega _{2}^{\varepsilon }}\varphi _{2}\psi _{2}\hspace{0.01in}%
\mathrm{d}x,\ \ \varphi =\left( \varphi _{1},\varphi _{2}\right) ,\ \ \psi
=\left( \psi _{1},\psi _{2}\right) , \\
&&\  \\
\left( \varphi ,\psi \right) _{V^{\varepsilon }} &=&\int_{\Omega
_{1}^{\varepsilon }}\nabla \varphi _{1}\nabla \psi _{1}\hspace{0.01in}%
\mathrm{d}x+\varepsilon ^{2}\int_{\Omega _{2}^{\varepsilon }}\nabla \varphi
_{2}\nabla \psi _{2}\hspace{0.01in}\mathrm{d}x+ \\
&&\varepsilon \int_{\Sigma ^{\varepsilon }}\left( \varphi _{1}-\varphi
_{2}\right) \left( \psi _{1}-\psi _{2}\right) \hspace{0.01in}\mathrm{d}%
\sigma ^{\varepsilon }
\end{eqnarray*}%
where $\mathrm{d}x$ denotes the Lebesgue measure on $\mathbb{R}^{d}$ and $%
\mathrm{d}\sigma ^{\varepsilon }$ the surfacic measure on $\Sigma
^{\varepsilon }$. The norms induced in $H^{\varepsilon }$ and $%
V^{\varepsilon }$ are denoted by $\left\Vert \cdot \right\Vert
_{H^{\varepsilon }}$ and $\left\Vert \cdot \right\Vert _{V^{\varepsilon }}$,
respectively. Clearly, $H^{\varepsilon }$ and $V^{\varepsilon }$\ are
Hilbert spaces when equipped with their respective norms. Moreover, it can
easily be shown that $V^{\varepsilon }$\ is separable, dense and
continuously embedded in $H^{\varepsilon }$.

Let us introduce the bilinear form $a^{\varepsilon }\left( \cdot ,\cdot
\right) :V^{\varepsilon }\times V^{\varepsilon }\longrightarrow \mathbb{R}$
defined by
\begin{eqnarray*}
a^{\varepsilon }\left( \varphi ,\psi \right) &=&\int_{\Omega
_{1}^{\varepsilon }}\alpha \nabla \varphi _{1}\nabla \psi _{1}\hspace{0.01in}%
\mathrm{d}x+\int_{\Omega _{2}^{\varepsilon }}\alpha _{b}\nabla \varphi
_{2}\nabla \psi _{2}\hspace{0.01in}\mathrm{d}x+ \\
&&\int_{\Sigma ^{\varepsilon }}\gamma ^{\varepsilon }\left( \varphi
_{1}-\varphi _{2}\right) \left( \psi _{1}-\psi _{2}\right) \hspace{0.01in}%
\mathrm{d}\sigma ^{\varepsilon }
\end{eqnarray*}%
where $\varphi =\left( \varphi _{1},\varphi _{2}\right) ,\ \psi =\left( \psi
_{1},\psi _{2}\right) \in V^{\varepsilon }$. We see that $a^{\varepsilon
}\left( \cdot ,\cdot \right) $ is continuous and uniformly coercive.

Let $\left( V^{\varepsilon }\right) ^{\prime }$ denote the dual space of $%
V^{\varepsilon }$. Let $\mathcal{A}^{\varepsilon }\in \mathcal{L}\left(
V^{\varepsilon },\left( V^{\varepsilon }\right) ^{\prime }\right) $ be given
by
\begin{equation*}
\mathcal{A}^{\varepsilon }\left( \varphi \right) \psi =a^{\varepsilon
}\left( \varphi ,\psi \right) ,\ \ \ \ \ \ \ \varphi ,\psi \in
V^{\varepsilon }.
\end{equation*}%
For convenience we shall denote $w^{\varepsilon }=\left( T^{\varepsilon
},T_{b}^{\varepsilon }\right) $,\ \ \ $g=\left( h,h_{b}\right) \ $and $\ $%
\begin{equation*}
f^{\varepsilon }=f\chi _{1}^{\varepsilon }+f_{b}\chi _{2}^{\varepsilon },\ \
\ \chi _{m}^{\varepsilon }\left( x\right) =\chi _{m}\left( \dfrac{x}{%
\varepsilon }\right) ,m=1,2.
\end{equation*}%
Let
\begin{equation*}
W^{1,2}\left( 0,T;H^{\varepsilon }\right) =\left\{ w\in L^{2}\left(
0,T;H^{\varepsilon }\right) :w^{\prime }=\frac{dw}{dt}\in L^{2}\left(
0,T;H^{\varepsilon }\right) \right\} .
\end{equation*}%
The variational formulation for (\ref{dp4p1})-(\ref{dp4p7}) reads as
follows: find $w^{\varepsilon }\in L^{2}\left( 0,T;V^{\varepsilon }\right) $
such that for every $\varphi \in W^{1,2}\left( 0,T;H^{\varepsilon }\right)
\cap L^{2}\left( 0,T;V^{\varepsilon }\right) $ with $\varphi \left( T\right)
=0$, we have
\begin{eqnarray}
&&-\int_{0}^{T}\left( w^{\varepsilon }\left( t\right) ,\partial _{t}\varphi
\left( t\right) \right) _{H^{\varepsilon }}\hspace{0.01in}\mathrm{d}%
t+\int_{0}^{T}\mathcal{A}^{\varepsilon }\left( w^{\varepsilon }\left(
t\right) \right) \varphi \left( t\right) \hspace{0.01in}\mathrm{d}t  \notag
\\
&=&\int_{0}^{T}\left( f^{\varepsilon },\varphi \left( t\right) \right)
_{H^{\varepsilon }}\hspace{0.01in}\mathrm{d}t+\left( g,\varphi \left(
0\right) \right) _{H^{\varepsilon }}.  \label{dp4ie}
\end{eqnarray}%
where $\mathrm{d}t$ denotes the Lebesgue measure on $\left( 0,T\right) $.

Next, we state the existence and uniqueness result for (\ref{dp4ie}) the
proof of which is given in the next section.

\begin{theorem}
\label{dp4t1}Let $\varepsilon >0$ be a sufficiently small parameter. Then,
there exists a unique weak solution $w^{\varepsilon }\in L^{2}\left(
0,T;V^{\varepsilon }\right) $ of Problem (\ref{dp4ie}) and the following
energy estimate holds:%
\begin{equation}
\left\Vert w^{\varepsilon }\right\Vert _{L^{\infty }\left(
0,T;H^{\varepsilon }\right) }+\left\Vert w^{\varepsilon }\right\Vert
_{L^{2}\left( 0,T;V^{\varepsilon }\right) }\leq C.  \label{dp4ea}
\end{equation}
\end{theorem}

Now, we are ready to give the main result of this paper whose proof will be
given in the last section.

We define the overall temperature in the biological tissue region $\Omega
_{1}^{\varepsilon }\cup \Omega _{2}^{\varepsilon }$ by\textbf{\ }%
\begin{equation*}
u^{\varepsilon }\left( t,x\right) =\chi _{1}\left( \frac{x}{\varepsilon }%
\right) T^{\varepsilon }\left( t,x\right) +\chi _{2}\left( \frac{x}{%
\varepsilon }\right) T_{b}^{\varepsilon }\left( t,x\right) ,\ \text{a.e. }%
\left( t,x\right) \in Q.
\end{equation*}

\begin{theorem}
\label{dp4t2}There exists a subsequence of $\left( u^{\varepsilon }\right) $%
, still denoted $\left( u^{\varepsilon }\right) $ such that there exist $%
T\in L^{2}\left( 0,T;H_{0}^{1}\left( \Omega \right) \right) $ and $T_{b}\in
L^{2}\left( Q;H_{\#}^{1}\left( Y\right) \right) $ with

\begin{enumerate}
\item $\chi _{1}\left( \frac{x}{\varepsilon }\right) u^{\varepsilon }$
weakly converges to $\left\vert Y_{1}\right\vert T$ in $L^{2}\left( Q\right)
;$

\item $\chi _{2}\left( \frac{x}{\varepsilon }\right) u^{\varepsilon }$
weakly converges to $\int_{Y_{2}}T_{b}\mathrm{d}y$ in $L^{2}\left( Q\right)
; $

\item $T$ is a solution to the homogenized problem:
\begin{gather}
\partial _{t}T-\int_{0}^{t}\mathcal{H}\left( t-\tau \right) T\left( \tau
\right) \hspace{0.01in}d\tau -\mathrm{div}\left( \widetilde{A}\nabla
T\right) +\widetilde{\gamma }T=\mathcal{F}\text{ in }Q,  \notag \\
\   \label{h0} \\
T=\text{ }0\text{ on }S,  \label{h1} \\
T\left( 0,x\right) =\left\vert Y_{1}\right\vert h\left( x\right) ,\ x\in
\Omega  \label{h2}
\end{gather}%
where $\mathcal{H}$, $\widetilde{A}$, $\widetilde{\gamma }$ and $\mathcal{F}$
are respectively given by (\ref{h00}), (\ref{h04}), (\ref{h02}) and (\ref%
{h01});

\item The temperature $T_{b}$ is related to $T$ by:
\end{enumerate}
\end{theorem}

Finally, we end this section by noticing that (\ref{h0}) is an
integro-differential equation of Barbashin type.

\section{Some auxiliary lemmas and proof of Theorem \protect\ref{dp4t1}}

We begin this section with some standard lemmas needed for proving the
existence and uniqueness result and also for establishing uniform a priori
estimates that are specifically important when using compactness techniques .

\begin{lemma}
\label{dp4l1}There exists a constant $C>0$, independent of $\varepsilon $
such that for all $\varphi _{1}\in H_{0}^{1}\left( \Omega \right) \cap
H^{1}\left( \Omega _{1}^{\varepsilon }\right) $ we have%
\begin{equation}
\left\Vert \varphi _{1}\right\Vert _{0,\Omega _{1}^{\epsilon }}\leq
C\left\Vert \nabla \varphi _{1}\right\Vert _{0,\Omega _{1}^{\epsilon }}.
\label{dp4e2}
\end{equation}
\end{lemma}

\begin{proof}
See for instance G. Allaire and F. Murat\cite[Lemma A.4]{am}.
\end{proof}

\begin{lemma}
\label{dp4l2}There exists a constant $C>0$, independent of $\varepsilon $
such that for all $\varphi _{2}\in H^{1}\left( \Omega _{2}^{\varepsilon
}\right) $ we have
\begin{equation}
\left\Vert \varphi _{2}\right\Vert _{0,\Omega _{2}^{\epsilon }}^{2}\leq
C\left( \varepsilon ^{2}\left\Vert \nabla \varphi _{2}\right\Vert _{0,\Omega
_{2}^{\epsilon }}^{2}+\varepsilon \left\Vert \varphi _{2}\right\Vert
_{0,\Sigma ^{\varepsilon }}^{2}\right) .  \label{dp4e1}
\end{equation}
\end{lemma}

\begin{proof}
See C. Conca\cite[Lemma 6.1]{conc}.
\end{proof}

\begin{lemma}
\label{dp4l3}There exists a constant $C>0$, independent of $\varepsilon $
such that for all $\varphi \in H^{1}\left( \Omega _{1}^{\varepsilon }\right)
$ we have
\begin{equation}
\varepsilon \left\Vert \varphi \right\Vert _{0,\Sigma ^{\varepsilon
}}^{2}\leq C\left( \varepsilon ^{2}\left\Vert \nabla \varphi \right\Vert
_{0,\Omega _{1}^{\varepsilon }}^{2}+\left\Vert \varphi \right\Vert
_{0,\Omega _{1}^{\varepsilon }}^{2}\right) ,  \label{dp4e3}
\end{equation}%
and%
\begin{equation}
\varepsilon \left\Vert \varphi \right\Vert _{0,\Sigma ^{\varepsilon
}}^{2}\leq C\left( \left\Vert \nabla \varphi \right\Vert _{0,\Omega
_{1}^{\varepsilon }}^{2}\right) .  \label{dp4e7}
\end{equation}
\end{lemma}

\begin{proof}
Using the trace theorem on $Y_{1}$ (see for e.g. R. A. Adams and J. F.
Fournier\cite{ada}), we know that there exists a constant $C\left(
Y_{1}\right) >0$ such that for every $\psi \in H^{1}\left( Y_{1}\right) $
\begin{equation*}
\int_{\Sigma }\left\vert \psi \right\vert ^{2}\hspace{0.01in}\mathrm{d}%
\sigma \leq C\left( \int_{Y_{1}}\left\vert \nabla \psi \right\vert ^{2}%
\hspace{0.01in}\mathrm{d}y+\int_{Y_{1}}\left\vert \psi \right\vert ^{2}%
\hspace{0.01in}\mathrm{d}y\right) .
\end{equation*}%
Then, we change $y$ by $x/\varepsilon $ and we get
\begin{equation}
\varepsilon \int_{\Sigma ^{\epsilon k}}\left\vert \varphi \right\vert ^{2}%
\hspace{0.01in}\mathrm{d}\sigma ^{\varepsilon }\leq C\left( \varepsilon
^{2}\int_{Y_{1}^{\varepsilon k}}\left\vert \nabla \varphi \right\vert ^{2}%
\hspace{0.01in}\mathrm{d}x+\int_{Y_{1}^{\epsilon k}}\left\vert \varphi
\right\vert ^{2}\hspace{0.01in}\mathrm{d}x\right) ,\   \label{dp4eq1}
\end{equation}%
for every $\varphi \in H^{1}\left( Y_{1}^{\epsilon k}\right) $, where
\begin{equation*}
Y_{1}^{\epsilon k}=\varepsilon \left( k+\Omega _{1}^{\varepsilon }\right) ,\
\Gamma ^{\epsilon k}=\varepsilon \left( k+\Gamma \right) \ \ k\in \mathbb{Z}%
^{d}\text{.}
\end{equation*}%
Note that the constant $C$ appearing in (\ref{dp4eq1}) is the same for all $%
\varepsilon >0$ and for all $k\in \mathbb{Z}^{d}$. Now, by taking the sum of
the inequalities (\ref{dp4eq1}) over all the cells $Y_{1}^{\epsilon k}$\
contained in $\Omega $, we obtain (\ref{dp4e3}). In fact, on the part of the
cells which contain a portion of the exterior boundary, the estimate (\ref%
{dp4e3}) still holds true, since those cells lie at a distance $O\left(
\varepsilon \right) $. As $\varepsilon $ is sufficiently small, say $%
\varepsilon <1$, we have from (\ref{dp4e3}) that for all $\varphi \in
H^{1}\left( \Omega _{1}^{\varepsilon }\right) $
\begin{equation}
\varepsilon \int_{\Gamma ^{\epsilon }}\left\vert \varphi \right\vert ^{2}%
\hspace{0.01in}\mathrm{d}\sigma ^{\varepsilon }\leq C\left( \int_{\Omega
_{1}^{\epsilon }}\left\vert \nabla \varphi \right\vert ^{2}\hspace{0.01in}%
\mathrm{d}x+\int_{\Omega _{1}^{\epsilon }}\left\vert \varphi \right\vert ^{2}%
\hspace{0.01in}\mathrm{d}x\right)  \label{dp4e4}
\end{equation}%
and using the Friedrich inequality (\ref{dp4e2}) in (\ref{dp4e4}), we get (%
\ref{dp4e7}).
\end{proof}

\begin{proof}[Proof of Theorem \protect\ref{dp4t1}]
We shall use the Lions Lemma (see for instance R. Showalter\cite[Prop. 2.3.,
Chap. III]{sb}). Since $a^{\varepsilon }\left( \cdot ,\cdot \right) $ is
coercive and continuous, it only remains to prove the continuity of the
form:
\begin{equation*}
\varphi =\left( \varphi _{1},\varphi _{2}\right) \mapsto L^{\varepsilon
}\left( \left( \varphi _{1},\varphi _{2}\right) \right) =\int_{0}^{T}\left(
h^{\varepsilon },\varphi \left( t\right) \right) _{H^{\varepsilon }}\hspace{%
0.01in}\mathrm{d}t
\end{equation*}%
on $L^{2}\left( 0,T;\left( V^{\varepsilon }\right) ^{\prime }\right) $.
First, using Cauchy-Schwarz inequality and (\ref{dp4h7}), we see that for
all $\varphi =\left( \varphi _{1},\varphi _{2}\right) \in V^{\varepsilon }$,
\begin{eqnarray}
\left\vert L^{\varepsilon }\left( \left( \varphi _{1},\varphi _{2}\right)
\right) \right\vert &=&\left\vert \int_{0}^{T}\left( \int_{\Omega
_{1}^{\epsilon }}f_{1}\varphi _{1}\hspace{0.01in}\mathrm{d}x+\int_{\Omega
_{2}^{\epsilon }}f_{2}\varphi _{2}\hspace{0.01in}\mathrm{d}x\right)
\right\vert  \notag \\
&\leq &M\left( f_{1},f_{2}\right) \left( ||\varphi _{1}||_{0,\Omega
_{T}^{\varepsilon }}+||\varphi _{2}||_{0,\Omega _{T}^{\varepsilon }}\right)
\label{dp4e9}
\end{eqnarray}%
where
\begin{equation*}
M\left( f_{1},f_{2}\right) =\max \left( ||f_{1}||_{0,\Omega
_{T}},||f_{2}||_{0,\Omega _{T}}\right) <+\infty
\end{equation*}%
is a constant independent of $\varepsilon $. Observe that $f\in L^{2}\left(
0,T;H^{\varepsilon }\right) $. Next, from (\ref{dp4e1}), we get
\begin{eqnarray}
\int_{\Omega _{2}^{\epsilon }}\left\vert \varphi _{2}\right\vert ^{2}\hspace{%
0.01in}\mathrm{d}x &\leq &C\left( \varepsilon ^{2}\int_{\Omega
_{2}^{\epsilon }}\left\vert \nabla \varphi _{2}\right\vert ^{2}\hspace{0.01in%
}\mathrm{d}x+\varepsilon \int_{\Gamma ^{\epsilon }}\left\vert \varphi
_{1}-\varphi _{2}\right\vert ^{2}\hspace{0.01in}\mathrm{d}\sigma
^{\varepsilon }\right.  \notag \\
&&\left. +\varepsilon \int_{\Gamma ^{\epsilon }}\left\vert \varphi
_{1}\right\vert ^{2}\hspace{0.01in}\mathrm{d}\sigma ^{\varepsilon }\right) .
\label{dp4e5}
\end{eqnarray}%
Now, combining (\ref{dp4e7}) and (\ref{dp4e5}) give
\begin{eqnarray*}
\int_{\Omega _{2}^{\epsilon }}\left\vert \varphi _{2}\right\vert ^{2}\hspace{%
0.01in}\mathrm{d}x &\leq &C\left( \int_{\Omega _{1}^{\epsilon }}\left\vert
\nabla \varphi _{1}\right\vert ^{2}\hspace{0.01in}\mathrm{d}x+\varepsilon
^{2}\int_{\Omega _{2}^{\epsilon }}\left\vert \nabla \varphi _{2}\right\vert
^{2}\hspace{0.01in}\mathrm{d}x\right. \\
&&\left. +\varepsilon \int_{\Gamma ^{\epsilon }}\left\vert \varphi
_{1}-\varphi _{2}\right\vert ^{2}\hspace{0.01in}\mathrm{d}\sigma
^{\varepsilon }\right) .
\end{eqnarray*}%
which means that
\begin{equation}
\int_{\Omega _{T}^{\epsilon }}\left\vert \varphi _{2}\right\vert ^{2}\hspace{%
0.01in}\mathrm{d}x\leq C\left\Vert \left( \varphi _{1},\varphi _{2}\right)
\right\Vert _{V^{\epsilon }}^{2}.  \label{dp4e8}
\end{equation}%
Observe that (\ref{dp4e1}) yields
\begin{equation}
\int_{\Omega _{T}^{\epsilon }}\left\vert \varphi _{1}\right\vert ^{2}\hspace{%
0.01in}\mathrm{d}x\leq C\left\Vert \left( \varphi _{1},\varphi _{2}\right)
\right\Vert _{V^{\epsilon }}^{2}.  \label{dp4e6}
\end{equation}%
Using (\ref{dp4e9}), (\ref{dp4e8}) and (\ref{dp4e6}) we deduce that
\begin{equation}
\left\vert L^{\varepsilon }\left( \left( \varphi _{1},\varphi _{2}\right)
\right) \right\vert \leq C\left\Vert \left( \varphi _{1},\varphi _{2}\right)
\right\Vert _{V^{\epsilon }}.  \label{lc1}
\end{equation}%
Thus, $L^{\varepsilon }$ is continuous on $L^{2}\left( 0,T;V^{\varepsilon
}\right) $. Note that the constant $C$ appearing in (\ref{dp4e9}) is
independent of $\varepsilon $.

By Lions Lemma, we conclude that there exists a unique solution $\left(
T^{\varepsilon },T_{b}^{\varepsilon }\right) \in L^{2}\left(
0,T;V^{\varepsilon }\right) $ to the weak formulation of (\ref{dp4p1})-(\ref%
{dp4p7}). Finally, putting $\left( \varphi _{1},\varphi _{2}\right) =\left(
T^{\varepsilon },T_{b}^{\varepsilon }\right) $ in (\ref{dp4ie}), using the
uniform coerciveness of $a^{\varepsilon }\left( \cdot ,\cdot \right) $, the
continuity of $L^{\varepsilon }$ and the Gronwall inequality yield the
uniform estimate (\ref{dp4ea}). This concludes the proof of the Theorem.
\end{proof}

\begin{remark}
\label{dp4r2}If $h=\left( h,h_{b}\right) $ is given in $V^{\varepsilon }$
then one can easily see that $w^{\varepsilon }=\left( T^{\varepsilon
},T_{b}^{\varepsilon }\right) \in W^{1,2}\left( 0,T;H^{\varepsilon }\right) $
and therefore $w^{\varepsilon }\in C\left( 0,T;V^{\varepsilon }\right) $.
\end{remark}

\

\section{The Homogenization procedure\label{s3}}

We shall first use the formal two-scale method (see for example A.
Bensoussan \& \textit{al.}\cite{blp} and E. Sanchez-Palencia\cite{san}) to
derive the homogenized system of (\ref{dp4p1})-(\ref{dp4p7}). To this end,
let us assume the following formal expansions for the two temperatures:
\begin{eqnarray}
T^{\varepsilon }\left( t,x\right) &=&T_{0}\left( t,x,y\right) +\varepsilon
T_{1}\left( t,x,y\right) +\varepsilon ^{2}T_{2}\left( t,x,y\right) +\ldots
\label{ex1} \\
T_{b}^{\varepsilon }\left( t,x\right) &=&T_{b}\left( t,x,y\right)
+\varepsilon T_{b1}\left( t,x,y\right) +\varepsilon ^{2}T_{b2}\left(
t,x,y\right) +\ldots  \label{ex2}
\end{eqnarray}%
where $y=x/\varepsilon $ is the microscopic variable and $\ T_{k}\left(
\cdot ,\cdot ,y\right) ,T_{bk}\left( \cdot ,\cdot ,y\right) ,\ldots $ ($%
k=0,1,2,\cdots $) are smooth unknown functions that are $Y$-periodic in the
third variable $y$. The idea of the two-scale method is to plug the above
asymptotic expansions (\ref{ex1})-(\ref{ex2}) into the set of equations (\ref%
{dp4p1})--(\ref{dp4p7}) and to identify powers of $\varepsilon $. This
yields a hierarchy of initial boundary value problems for the successive
terms $T_{k},T_{bk}$.

\begin{notation}
In what follows, the subscript $x$, $y$ on a differential operator denotes
the derivative with respect to $x$, $y$ respectively.
\end{notation}

At the first step, Equation (\ref{dp4p1}) at $\varepsilon ^{-2}$ order and
Equation (\ref{dp4p3}) at $\varepsilon ^{-1}$ order give
\begin{equation}
-\mathrm{div}_{y}\left( \alpha \nabla _{y}T_{0}\right) =0\text{ in }Q\times
Y_{1}  \label{2m}
\end{equation}%
and
\begin{equation}
\alpha \nabla _{y}T_{0}\cdot \nu =0\text{ on }Q\times \Sigma \text{,}
\label{2mm}
\end{equation}%
where $\nu $ is the unit outward normal to $\Sigma .$ Testing (\ref{2m}) by $%
\zeta \in H_{\#}^{1}\left( Y_{1}\right) $, integrating by parts on $Y_{1}$,
taking into account (\ref{2mm}) and the $Y$-periodicity of $\alpha \nabla
_{y}T_{0},$ we get the following weak formulation:%
\begin{equation*}
\left\{
\begin{array}{l}
a_{Y_{1}}\left( T_{0},\zeta \right) \overset{\mathrm{def}}{=}%
\int_{Y_{1}}\alpha \nabla _{y}T_{0}\cdot \nabla _{y}\zeta \hspace{0.01in}%
\mathrm{d}y=0\ \ \text{for every }\zeta \in H_{\#}^{1}\left( Y_{1}\right) /%
\mathbb{R}\text{,} \\
\  \\
T_{0}\in H_{\#}^{1}\left( Y_{1}\right) /\mathbb{R}\text{.}%
\end{array}%
\right.
\end{equation*}%
The bilinear form $a_{Y_{1}}$ is clearly continuous and coercive on $%
H_{\#}^{1}\left( Y_{1}\right) /\mathbb{R}$ and therefore standard results on
uniformly elliptic equations in periodic domain (A. Bensoussan \& \textit{al.%
}\cite{blp}) yields that $T$ is independent of the periodic variable $y$,
namely there exist $T\left( t,x\right) $
\begin{equation*}
T_{0}\left( t,x,y\right) =T\left( t,x\right) ,\ \ \ \ \text{for a.e. }t\in
\left( 0,T\right) \text{ and \ \ }x\in \Omega \text{.}
\end{equation*}%
Next, in the second step, Equation (\ref{dp4p1}) at $\varepsilon ^{-1}$,
Equation (\ref{dp4p3}) at $\varepsilon ^{0}$ orders give the following
corrector problem:
\begin{eqnarray}
&&-\alpha \Delta _{y}T_{1}=0\ \text{in }Q\times Y_{1},  \label{fv4} \\
&&\left( \alpha \nabla _{y}T_{1}\right) \cdot \nu =-\left( \alpha \nabla
T\right) \cdot \nu \text{ on }Q\times \Sigma ,  \label{fv5} \\
&&y\longmapsto T_{1}\left( x,y\right) \text{ }Y-\text{periodic, }\left(
t,x\right) \in Q.  \label{fv6}
\end{eqnarray}%
The corresponding weak formulation is given by
\begin{equation*}
\text{(}\mathcal{P}\text{)\ }\left\{
\begin{array}{l}
a_{Y_{1}}\left( T_{1},\zeta \right) =-a_{Y_{1}}\left( T,\zeta \right) ,\ \
\forall \zeta \in H_{\#}^{1}\left( Y_{1}\right) /\mathbb{R}\text{,} \\
\  \\
T_{1}\in H_{\#}^{1}\left( Y_{1}\right) /\mathbb{R}\text{.}%
\end{array}%
\right.
\end{equation*}%
As before, there exists a unique solution $T_{1}\left( t,x,\cdot \right) \in
H_{\#}^{1}\left( Y_{1}\right) /\mathbb{R}$ of problem ($\mathcal{P}$) which
can be computed as follows. Let us consider for $1\leq i\leq d,$ the
following cell problems:
\begin{equation*}
\text{(}\mathcal{P}_{i}\text{)\ }\left\{
\begin{array}{l}
a_{Y_{1}}\left( \omega _{i},\zeta \right) =-a_{Y_{1}}\left( e_{i},\zeta
\right) ,\ \ \forall \zeta \in H_{\#}^{1}\left( Y_{1}\right) /\mathbb{R}, \\
\  \\
\omega _{i}\in H_{\#}^{1}\left( Y_{1}\right) /\mathbb{R}%
\end{array}%
\right.
\end{equation*}%
where $\left( e_{i}\right) _{1\leq i\leq d}$ is the canonical basis. These
problems ($\mathcal{P}_{i}$) are obtained from ($\mathcal{P}$) by replacing $%
\nabla T$ with the vector $e_{i}$, $i=1,2,\cdots ,d$. It follows that for
each $i$ Problem ($\mathcal{P}_{i}$) admits a unique solution $\omega
_{i}\in H_{\#}^{1}\left( Y_{1}\right) /\mathbb{R}$. Furtheremore, thanks to
the linearity of ($\mathcal{P}$), we may write that:
\begin{equation}
T_{1}\left( t,x,y\right) =\sum_{i=1}^{d}\frac{\partial T}{\partial x_{i}}%
\left( t,x\right) \omega _{i}\left( y\right) +\tilde{u}\left( t,x\right)
\label{3m}
\end{equation}%
for a.e. $\left( t,x,y\right) \in Q\times Y_{1}$ and where $\tilde{u}\left(
t,x\right) $ is any additive constant.

At the final step, Equation (\ref{dp4p1})-(\ref{dp4p2}) at $\varepsilon ^{0}$%
, Equations (\ref{dp4p3})-(\ref{dp4p4}) at $\varepsilon ^{1}$ orders yield
the following initial boundary-value problem:
\begin{gather}
-\alpha \Delta _{y}T_{2}=f-\partial _{t}T+\alpha \mathrm{div}_{y}\left(
\nabla _{x}T_{1}\right) +  \notag \\
\ \ \ \ \ \ \ \ \ \ \ \ \ \ \ \alpha \mathrm{div}_{x}\left( \left( \nabla
_{y}T_{1}+\nabla _{x}T\right) \right) \ \text{in }Q\times Y_{1},  \label{4o}
\end{gather}%
\begin{equation}
\partial _{t}T_{b}-\alpha _{b}\Delta _{y}T_{b}=f_{b}\text{ in }Q\times Y_{2},
\label{4n}
\end{equation}%
\begin{equation}
\left( \alpha \nabla _{y}T_{2}\right) \cdot \nu =-\left( \alpha \nabla
_{x}T_{1}\right) \cdot \nu +\left( \alpha _{b}\nabla _{y}T_{b}\right) \cdot
\nu \text{ on }Q\times \Sigma ,  \label{4p}
\end{equation}%
\begin{equation}
\alpha \nabla _{y}T_{2}\cdot \nu =-\alpha \nabla _{x}T_{1}\cdot \nu +\gamma
\left( T-T_{b}\right) \text{ on }Q\times \Sigma ,  \label{4q}
\end{equation}%
\begin{eqnarray}
y &\mapsto &T_{2}\left( t,x,y\right) \text{ }Y-\text{periodic,}  \label{4r}
\\
y &\mapsto &\ T_{b}\left( t,x,y\right) \text{ }Y-\text{periodic.}  \label{4u}
\end{eqnarray}

The weak formulation of Equations (\ref{4o})--(\ref{4r}) is \
\begin{equation*}
(\mathcal{P})\left\{
\begin{array}{l}
a_{Y_{1}}\left( T_{2},\zeta \right) =\left\langle F,\zeta \right\rangle \ \
\text{for every }\zeta \in H_{\#}^{1}\left( Y_{1}\right) /\mathbb{R}\text{,}
\\
\  \\
T_{2}\in H_{\#}^{1}\left( Y_{1}\right) /\mathbb{R}%
\end{array}%
\right.
\end{equation*}%
where
\begin{eqnarray*}
\left\langle F,\psi \right\rangle &=&\left( \int_{Y_{1}}-\zeta \hspace{0.01in%
}\mathrm{d}y\right) \partial _{t}T+\int_{Y_{1}}\mathrm{div}_{x}\left( \alpha
\left( \nabla _{y}T_{1}+\nabla T\right) \right) \zeta \hspace{0.01in}\mathrm{%
d}y \\
&&-\int_{Y_{1}}\alpha \nabla _{x}T_{1}\cdot \nabla _{y}\zeta \hspace{0.01in}%
\mathrm{d}y+\int_{Y_{1}}f\zeta \ \mathrm{d}y+\int_{\Sigma }\gamma \left(
T-T_{b}\right) \zeta \hspace{0.01in}\mathrm{d}\sigma
\end{eqnarray*}%
Using the divergence Theorem ( as in E. Sanchez-Palencia\cite{san}), a
necessary condition for the existence of $T_{2}$ is that $\left\langle
F,1\right\rangle =0$, namely
\begin{eqnarray}
&&\int_{Y_{1}}\left( -\partial _{t}T+\mathrm{div}_{x}\left( \alpha \left(
\nabla _{y}T_{1}+\nabla T\right) \right) \right) \hspace{0.01in}\mathrm{d}y
\notag \\
&&+\int_{\Sigma }\gamma \left( T-T_{b}\right) \hspace{0.01in}\mathrm{d}%
\sigma =\left\vert Y_{1}\right\vert f.  \label{5a}
\end{eqnarray}%
Using (\ref{3m}), equation (\ref{5a}) becomes
\begin{equation}
\left\vert Y_{1}\right\vert \partial _{t}T-\mathrm{div}\left( \widetilde{A}%
\nabla T\right) +\int_{\Sigma }\gamma \left( T-T_{b}\right) \hspace{0.01in}%
\mathrm{d}\sigma =\left\vert Y_{1}\right\vert f,  \label{hs1}
\end{equation}%
where $\left\vert Y_{1}\right\vert $ stands for the volume of $Y_{1}$, The
matrix $\widetilde{A}$ is given by
\begin{equation}
\widetilde{A}=\left( \tilde{a}_{ij}\right) _{1\leq i,j\leq d}\text{,\ \ \ }%
\tilde{a}_{ij}=\int_{Y_{1}}\alpha \left( \nabla _{y}\omega _{i}+e_{i}\right)
\cdot \left( \nabla _{y}\omega _{j}+e_{j}\right) \hspace{0.01in}\mathrm{d}y%
\text{.}  \label{h04}
\end{equation}%
Equation (\ref{hs1}) is the so-called macroscopic equation for the
temperature $T$. The boundary condition for $T$ is obtained from (\ref{dp4p5}%
) at $\varepsilon ^{0}$ order and it reads
\begin{equation}
T=0\text{ on }\partial \Omega \text{.}  \label{b1}
\end{equation}%
Similarly, Equations (\ref{dp4p6})-(\ref{dp4p7}) give the initial conditions
for $T$ and $T_{b}$:
\begin{eqnarray}
&&T\left( 0,x\right) =\left\vert Y_{1}\right\vert h\left( x\right) ,\ x\in
\Omega ,  \label{4s} \\
&&T_{b0}\left( 0,x,y\right) =\chi _{2}\left( y\right) h_{b}\left( x\right)
,\ x\in \Omega ,\ \ y\in Y_{2}  \label{4t}
\end{eqnarray}

Next, we proceed further and focus our attention on fluid temperature $%
T_{b}\left( t,x,y\right) $. From (\ref{4p}) and (\ref{4q}) it is easily seen
that
\begin{equation}
\alpha _{b}\nabla _{y}T_{b}\cdot \nu =\gamma \left( T-T_{b}\right) \text{ on
}Q\times \Sigma .  \label{hs2}
\end{equation}%
It is easily shown that (\ref{4n}), (\ref{hs2}) and (\ref{4t}) admits a
unique weak solution $T_{b}\in H_{\#}^{1}\left( Y_{2}\right) $. Moreover,
denoting%
\begin{equation*}
b\left( \zeta ,\eta \right) =\int_{Y_{2}}\alpha _{b}\nabla _{y}\zeta \cdot
\nabla _{y}\eta \hspace{0.01in}\mathrm{d}y+\int_{\Sigma }\gamma \zeta \eta
\hspace{0.01in}\mathrm{d}\sigma ,\ \ \ \zeta ,\eta \in H^{1}\left(
Y_{2}\right)
\end{equation*}%
and $\mathcal{B\ }:\ H^{1}\left( Y_{2}\right) \longrightarrow H^{1}\left(
Y_{2}\right) $ defined by $\left\langle \mathcal{B}\left( \zeta \right)
,\eta \right\rangle =b\left( \zeta ,\eta \right) $, and applying (as in U.
Hornung\cite{hor}) the Duhamel's principle to equations (\ref{4n}), (\ref%
{hs2}) and (\ref{4t}), The leading term $T_{b}$ ca be thus decomposed as the
sum of three terms:\textbf{\ }%
\begin{eqnarray}
T_{b}\left( \tau ,x,y\right) &=&T_{bi}\left( \tau ,x,y\right)
+\int_{0}^{\tau }\partial _{\tau }\omega \left( \tau -t,y\right) T\left(
t,x\right) \hspace{0.01in}\mathrm{d}t  \notag \\
&&+\int_{0}^{\tau }\partial _{\tau }\mu \left( \tau -t,y\right) f_{b}\left(
t,x\right) \hspace{0.01in}\mathrm{d}t  \label{hs8}
\end{eqnarray}%
where $T_{bi}$ is the evolution of the initial temperature $h_{b}$. It is
given by $T_{bi}\left( t,x,y\right) =\mathrm{e}^{-t\mathcal{B}}h_{b}\left(
x,y\right) =\theta \left( t,y\right) h_{b}\left( x\right) $ where $\theta
\left( t,y\right) $ is the unique of the weak solution of the cell problem:
\begin{eqnarray*}
&&\partial _{t}\theta -\alpha _{b}\Delta _{y}\theta =0\text{ in }\left(
0,T\right) \times Y_{2}, \\
&&\alpha _{b}\nabla _{y}\theta \cdot \nu +\gamma \theta =0\text{ on }\left(
0,T\right) \times \Sigma , \\
&&y\longmapsto \theta \left( t,y\right) \text{ }Y-\text{periodic,} \\
&&\theta \left( 0,y\right) =1\text{ in }\Omega \times Y_{2}\text{.}
\end{eqnarray*}%
On the other hand $\omega ,$ $\sigma $\ are respectively the unique weak
solutions of the following cell problems:
\begin{eqnarray*}
&&\partial _{t}\omega -\alpha _{b}\Delta _{y}\omega =0\text{ in }\left(
0,T\right) \times Y_{2}, \\
&&\alpha _{b}\nabla _{y}\omega \cdot \nu +\gamma \omega =\gamma \text{ on }%
\left( 0,T\right) \times \Sigma , \\
&&y\longmapsto \omega \left( t,y\right) \text{ }Y-\text{periodic,} \\
&&\omega \left( 0,y\right) =0\text{ in }\Omega \times Y_{2}\text{,}
\end{eqnarray*}%
and
\begin{eqnarray*}
&&\partial _{t}\sigma -\alpha _{b}\Delta _{y}\sigma =1\text{ in }\left(
0,T\right) \times Y_{2}, \\
&&\alpha _{b}\nabla _{y}\sigma \cdot \nu +\gamma \sigma =0\text{ on }\left(
0,T\right) \times \Sigma , \\
&&y\longmapsto \sigma \left( t,y\right) \text{ }Y-\text{periodic,} \\
&&\sigma \left( 0,y\right) =0\text{ in }\Omega \times Y_{2}\text{.}
\end{eqnarray*}

Inserting (\ref{hs8}) into (\ref{hs1}) we get the homogenized
integro-differential equation of Barbashin type for the temperature $T$ (see
(\ref{h0})-(\ref{h2})):%
\begin{eqnarray*}
&&\partial _{t}T-\int_{0}^{t}\mathcal{H}\left( t-\tau \right) T\left( \tau
\right) \hspace{0.01in}d\tau -\mathrm{div}\left( \widetilde{A}\nabla
T\right) +\widetilde{\gamma }T=\mathcal{F}\text{ in }Q, \\
&&T=\text{ }0\text{ on }S, \\
&&T\left( 0,x\right) =\left\vert Y_{1}\right\vert h\left( x\right) ,\ x\in
\Omega
\end{eqnarray*}%
where%
\begin{equation}
\widetilde{\gamma }=\frac{1}{\left\vert Y_{1}\right\vert }\int_{\Sigma
}\gamma \ \mathrm{d}\sigma  \label{h02}
\end{equation}%
and, where $\mathcal{H}$ and $\mathcal{F}$ are given by
\begin{eqnarray}
\mathcal{H}\left( \tau ,x\right) &=&\frac{1}{\left\vert Y_{1}\right\vert }%
\int_{\Sigma }\gamma (y)\partial _{t}\omega \left( \tau ,y\right) \hspace{%
0.01in}\mathrm{d}\sigma ,  \label{h00} \\
\mathcal{F}\left( \tau ,x\right) &=&f\mathbf{(}x\mathbf{)}+\frac{1}{%
\left\vert Y_{1}\right\vert }\left( \int_{\Sigma }\gamma \left( y\right)
T_{bi}\left( \tau ,x,y\right) \hspace{0.01in}\mathrm{d}\sigma +\right.
\label{h01} \\
&&\left. \int_{0}^{\tau }\int_{\Sigma }\partial _{\tau }\sigma \left( \tau
-t,y\right) f_{b}\left( \tau ,y\right) \hspace{0.01in}\mathrm{d}\sigma
\mathrm{d}t\right) \text{.}  \notag
\end{eqnarray}

\section{Proof of Theorem \protect\ref{dp4t2}}

In this section, we shall derive the homogenized system (\ref{h0})-(\ref{h2}%
). To do so, we shall use the two-scale convergence technique that we recall
hereafter.

We shall first begin with some notations. We define $\mathcal{C}_{\#}(Y)$ to
be the space of all continuous functions on $\mathbb{R}^{d}$ which are $Y$%
-periodic. Let $\mathcal{C}_{\#}^{\infty }(Y)=\mathcal{C}^{\infty }(\mathbb{R%
}^{d})\cap \mathcal{C}_{\#}(Y)$ and let $L_{\#}^{2}\left( Y\right) $ (resp. $%
L_{\#}^{2}\left( Y_{m}\right) $, $m=1,2$) to be the space of all functions
belonging to $L_{\mathrm{loc}}^{2}\left( \mathbb{R}^{d}\right) $ (resp. $L_{%
\mathrm{loc}}^{2}\left( Z_{m}\right) $) which are $Y$-periodic, and $%
H_{\#}^{1}\left( Y\right) $ (resp. $H_{\#}^{1}\left( Y_{m}\right) $) to be
the space of those functions together with their derivatives belonging to $%
L_{\#}^{2}\left( Y\right) $ (resp. $L_{\#}^{2}\left( Z_{m}\right) $).

Now, we recall the definition and main results of the two-scale convergence
method. For more details, we refer the reader to G. Allaire\cite{all}.

\begin{definition}
\label{dp4d2}A sequence $\left( \vartheta ^{\varepsilon }\right) \ $in $%
L^{2}\left( \Omega \right) $ two-scale converges to $\vartheta \in
L^{2}\left( \Omega \times Y\right) $ (we write $\vartheta ^{\varepsilon }%
\overset{2-s}{\rightharpoonup }\vartheta $) if, for any admissible test
function $\varphi \in L^{2}\left( \Omega ;\mathcal{C}_{\#}(Y)\right) $, we
have
\begin{equation*}
\lim_{\varepsilon \rightarrow 0}\int_{\Omega }\vartheta ^{\varepsilon
}\left( x\right) \varphi \left( x,\frac{x}{\varepsilon }\right) \hspace{%
0.01in}\mathrm{d}x=\int_{\Omega \times Y}\vartheta \left( x,y\right) \varphi
\left( x,y\right) \hspace{0.01in}\mathrm{d}x\mathrm{d}y\text{.}
\end{equation*}
\end{definition}

\begin{theorem}
\label{dp4t3}Let $(\vartheta ^{\varepsilon })$ be a sequence of functions in
$L^{2}(\Omega )$. Assume that $(\vartheta ^{\varepsilon })$ is uniformly
bounded. Then, there exist $\vartheta \in L^{2}(\Omega \times Y)$ and a
subsequence of $(\vartheta ^{\varepsilon })$ which two-scale converges\ to $%
T_{b}$.
\end{theorem}

\begin{theorem}
\label{dp4t4}Let $(\vartheta ^{\varepsilon })$ be a uniformly bounded
sequence in $H^{1}(\Omega )$ (resp. $H_{0}^{1}(\Omega )$). Then, up to a
subsequence, there exist $\vartheta \in H^{1}\left( \Omega \right) $ (resp. $%
H_{0}^{1}(\Omega )$) and $\vartheta _{0}\in L^{2}(\Omega ;H_{\#}^{1}(Y)/%
\mathbb{R})$ such that
\begin{equation*}
\vartheta ^{\varepsilon }\overset{2-s}{\rightharpoonup }\vartheta ;\ \ \ \ \
\nabla \vartheta ^{\varepsilon }\overset{2-s}{\rightharpoonup }\nabla
\vartheta +\nabla _{y}\vartheta _{0}.
\end{equation*}
\end{theorem}

The following result will be of use, see G. Allaire \& \textit{al.}\cite[%
Proposition 2.6]{adh}.

\begin{theorem}
\label{dp4t5}Let $(\vartheta ^{\varepsilon })$ be a sequence of functions in
$H^{1}(\Omega )$ such that
\begin{equation*}
\left\Vert \vartheta ^{\varepsilon }\right\Vert _{L^{2}\left( \Omega \right)
}+\varepsilon \left\Vert \nabla \vartheta ^{\varepsilon }\right\Vert
_{L^{2}\left( \Omega \right) ^{3}}\leq C\text{.}
\end{equation*}%
Then there exist $\vartheta \in L^{2}\left( \Omega ;H_{\#}^{1}(Y)\right) $
and a subsequence of $\left( \vartheta ^{\varepsilon }\right) $, still
denoted by $\left( \vartheta ^{\varepsilon }\right) $, such that
\begin{equation*}
\vartheta ^{\varepsilon }\overset{2-s}{\rightharpoonup }\vartheta ,\ \ \ \ \
\varepsilon \nabla \vartheta ^{\varepsilon }\overset{2-s}{\rightharpoonup }%
\nabla _{y}\vartheta
\end{equation*}%
and for every $\varphi \in \mathcal{D}\left( \Omega ;\mathcal{C}%
_{\#}(Y)\right) $ we have:%
\begin{equation*}
\lim_{\varepsilon \rightarrow 0}\int_{\Sigma ^{\varepsilon }}\vartheta
^{\varepsilon }\left( x\right) \varphi \left( x,\frac{x}{\varepsilon }%
\right) \hspace{0.01in}\mathrm{d}\sigma ^{\varepsilon }=\int_{\Omega \times
\Sigma }\vartheta \left( x,y\right) \varphi \left( x,y\right) \hspace{0.01in}%
\mathrm{d}x\mathrm{d}\sigma .
\end{equation*}
\end{theorem}

The notion of two-scale convergence can easily be extended to time-dependent
sequences without affecting the results stated above, namely Theorems \ref%
{dp4t3}, \ref{dp4t4} and \ref{dp4t5}. According to G.W. Clark and R.
Showalter\cite{cs}, we give the following:

\begin{definition}
\label{dp4d3}We say that a sequence $\left( \vartheta ^{\varepsilon }\right)
$ in $L^{2}\left( Q\right) $ two-scale converges to $\vartheta \in
L^{2}\left( Q\times Y\right) $ (we write $\vartheta ^{\varepsilon }\overset{%
2-s}{\rightharpoonup }\vartheta $) if, for any test function $\varphi \in
L^{2}\left( Q;\mathcal{C}_{\#}(Y)\right) $, we have%
\begin{equation*}
\underset{\varepsilon \rightarrow 0}{\lim }\int_{Q}\vartheta ^{\varepsilon
}\left( t,x\right) \varphi \left( t,x,\dfrac{x}{\varepsilon }\right) \hspace{%
0.01in}\mathrm{d}t\mathrm{d}x=\int_{Q\times Y}\vartheta \left( t,x,y\right)
\varphi \left( t,x,y\right) \hspace{0.01in}\mathrm{d}t\mathrm{d}x\mathrm{d}y.
\end{equation*}

\
\end{definition}

\begin{remark}
\label{dp4r3}If $\left( \vartheta ^{\varepsilon }\right) $ is a uniformly
bounded sequence in $L^{2}\left( Q\right) $, then there exists $\vartheta
\in L^{2}\left( Q\right) $ such that, up to a subsequence, $\vartheta
^{\varepsilon }\overset{2-s}{\rightharpoonup }\vartheta $ in the sense of
Def. \ref{dp4d3}. Moreover, if $\left( \vartheta ^{\varepsilon }\right) $ is
uniformly bounded in $L^{2}\left( 0,T;H^{1}\left( \Omega \right) \right) $,
then up to a subsequence, there exist $\vartheta \in L^{2}\left(
0,T;H^{1}\left( \Omega \right) \right) $ and $\vartheta _{0}\in L^{2}\left(
Q;H_{\#}^{1}\left( Y\right) /\mathbb{R}\right) $ such that $\vartheta
^{\varepsilon }\overset{2-s}{\rightharpoonup }\vartheta $ and $\nabla
\vartheta ^{\varepsilon }\overset{2-s}{\rightharpoonup }\nabla \vartheta
+\nabla _{y}\vartheta _{0}$. On the other hand, if a sequence $\left(
\vartheta ^{\varepsilon }\right) $ is such that
\begin{equation*}
\left\Vert \vartheta ^{\varepsilon }\right\Vert _{L^{2}\left( Q\right)
}+\varepsilon \left\Vert \nabla \vartheta ^{\varepsilon }\right\Vert
_{L^{2}\left( Q\right) }\leq C,
\end{equation*}%
then, up to a subsequence, there exists $\vartheta \in L^{2}\left(
0,T;H_{\#}^{1}\left( Y\right) \right) $ such that $\vartheta ^{\varepsilon }%
\overset{2-s}{\rightharpoonup }\vartheta $ and $\varepsilon \nabla \vartheta
^{\varepsilon }\overset{2-s}{\rightharpoonup }\nabla _{y}\vartheta $.
Furthermore, for every $\varphi \in \mathcal{D}\left( Q;\mathcal{C}%
_{\#}(Y)\right) $ we have:%
\begin{equation*}
\lim_{\varepsilon \rightarrow 0}\int_{S^{\varepsilon }}\vartheta
^{\varepsilon }\left( t,x\right) \varphi \left( t,x,\frac{x}{\varepsilon }%
\right) \hspace{0.01in}\mathrm{d}\sigma ^{\varepsilon }=\int_{Q\times \Sigma
}\vartheta \left( t,x,y\right) \varphi \left( t,x,y\right) \hspace{0.01in}%
\mathrm{d}x\mathrm{d}s
\end{equation*}%
where $\mathrm{d}s$ denotes the surface measure on $\Sigma $.
\end{remark}

Next we focus our attention on the two-scale convergence process, that is
deriving the two-scale homogenized system by employing the above compacity
results: Theorems \ref{dp4t3}-\ref{dp4t5} and Remark \ref{dp4r3}. To do
this, let us choose the following test functions: Let
\begin{equation*}
\varphi _{1}\in W^{1,2}\left( 0,T;\mathcal{D}\left( \Omega \right) \right)
,\ \psi \in W^{1,2}\left( 0,T;\mathcal{D}\left( \Omega ;\mathcal{C}%
_{\#}(Y)\right) \right)
\end{equation*}
and
\begin{equation*}
\varphi _{2}\in W^{1,2}\left( 0,T;\mathcal{D}\left( \Omega ;\mathcal{C}%
_{\#}(Y)\right) \right)
\end{equation*}
with
\begin{equation*}
\varphi _{1}\left( T,\cdot ,\cdot \right) =\varphi _{2}\left( T,\cdot ,\cdot
\right) =0\text{.}
\end{equation*}%
Set
\begin{equation*}
\varphi ^{\varepsilon }\left( t,x\right) =\left( \varphi _{1}\left(
t,x\right) +\varepsilon \psi \left( t,x,\frac{x}{\varepsilon }\right)
,\varphi _{2}\left( t,x,\frac{x}{\varepsilon }\right) \right) .
\end{equation*}%
Taking $\varphi =\varphi ^{\varepsilon }$ as a test function in (\ref{dp4ie}%
), we get
\begin{eqnarray}
&&-\int_{Q_{1}^{\varepsilon }}T^{\varepsilon }\partial _{t}\varphi
_{1}-\int_{Q_{2}^{\varepsilon }}T_{b}^{\varepsilon }\partial _{t}\varphi
_{2}+\int_{Q_{1}^{\varepsilon }}\alpha \nabla T^{\varepsilon }\left( \nabla
\varphi _{1}+\left( \nabla _{y}\psi \right) ^{\varepsilon }\right) +  \notag
\\
&&\varepsilon \int_{Q_{2}^{\varepsilon }}\alpha _{b}\nabla
T_{b}^{\varepsilon }\left( \nabla _{y}\varphi _{2}\right) ^{\varepsilon
}+\varepsilon \int_{S^{\varepsilon }}\gamma \left( T^{\varepsilon
}-T_{b}^{\varepsilon }\right) \left( \varphi _{1}-\varphi _{2}\right)  \notag
\\
&=&\int_{Q_{1}^{\varepsilon }}f\varphi _{1}+\int_{Q_{2}^{\varepsilon
}}f_{b}\varphi _{2}+\int_{\Omega _{1}^{\varepsilon }}h\varphi _{1}\left(
0\right) +\int_{\Omega _{2}^{\varepsilon }}h_{b}\varphi _{2}\left( 0\right)
+\varepsilon K^{\varepsilon }  \label{fv2}
\end{eqnarray}%
where%
\begin{eqnarray*}
K^{\varepsilon } &=&O\left( \varepsilon \right) =\int_{Q_{1}^{\varepsilon
}}T^{\varepsilon }\partial _{t}\psi \hspace{0.01in}\mathrm{d}t\mathrm{d}%
x-\int_{Q_{1}^{\varepsilon }}\alpha \nabla T^{\varepsilon }\left( \nabla
_{x}\psi \right) ^{\varepsilon }\hspace{0.01in}\mathrm{d}t\mathrm{d}x \\
&&+\int_{Q_{1}^{\varepsilon }}f\psi \hspace{0.01in}\mathrm{d}t\mathrm{d}%
x+\int_{\Omega _{1}^{\varepsilon }}h\psi \left( 0\right) \hspace{0.01in}%
\mathrm{d}x+\varepsilon \int_{S^{\varepsilon }}\gamma \left( T^{\varepsilon
}-T_{b}^{\varepsilon }\right) \psi \hspace{0.01in}\mathrm{d}t\mathrm{d}%
\sigma ^{\varepsilon }.
\end{eqnarray*}%
Now, thanks to the assumptions (\ref{dp4h7}) and to the a priori estimates (%
\ref{dp4ea}), using Theorems \ref{dp4t3}-\ref{dp4t5} and Remark \ref{dp4r3},
we have up to a subsequence, the following two scale convergences:

\begin{gather*}
\chi _{1}T^{\varepsilon }\overset{2-s}{\rightharpoonup }\chi _{1}T,\ \ \chi
_{2}T_{b}^{\varepsilon }\overset{2-s}{\rightharpoonup }\chi _{2}T_{b}, \\
\ \ \chi _{1}\nabla T^{\varepsilon }\overset{2-s}{\rightharpoonup }\chi
_{1}\left( \nabla T+\nabla _{y}T_{1}\right) ,\ \varepsilon \chi _{2}\nabla
T_{b}^{\varepsilon }\overset{2-s}{\rightharpoonup }\chi _{2}\nabla _{y}T_{b},
\\
\lim_{\varepsilon \rightarrow 0}\int_{S^{\varepsilon }}T^{\varepsilon
}\varphi \left( t,x,\frac{x}{\varepsilon }\right) \hspace{0.01in}\mathrm{d}%
\sigma ^{\varepsilon }=\int_{Q\times \Sigma }T\varphi \left( t,x,y\right)
\hspace{0.01in}\mathrm{d}x\mathrm{d}s, \\
\lim_{\varepsilon \rightarrow 0}\int_{S^{\varepsilon }}T_{b}^{\varepsilon
}\varphi \left( t,x,\frac{x}{\varepsilon }\right) \hspace{0.01in}\mathrm{d}%
\sigma ^{\varepsilon }=\int_{Q\times \Sigma }T_{b}\varphi \left(
t,x,y\right) \hspace{0.01in}\mathrm{d}x\mathrm{d}s,
\end{gather*}%
where $T\in L^{2}\left( 0,T;H_{0}^{1}\left( \Omega \right) \right) ,\
T_{b}\in L^{2}\left( Q;H_{\#}^{1}\left( Y\right) \right) $ and $T_{1}\in
L^{2}\left( Q;H_{\#}^{1}\left( Y\right) /\mathbb{R}\right) $. Now, passing
to the limit in (\ref{fv2}) and taking into account the above limits yield
the two scale system:%
\begin{gather}
-\int_{Q\times Y_{1}}T\partial _{t}\varphi _{1}-\int_{Q\times
Y_{2}}T_{b}\partial _{t}\varphi _{2}+\int_{Q\times Y_{1}}\alpha \left(
\nabla T+\nabla _{y}T_{1}\right) \left( \nabla \varphi _{1}+\nabla _{y}\psi
\right)  \notag \\
+\int_{Q\times Y_{2}}\alpha _{b}\nabla _{y}T_{b}\nabla _{y}\varphi
_{2}+\int_{Q\times \Sigma }\gamma \left( T-T_{b}\right) \left( \varphi
_{1}-\varphi _{2}\right)  \notag \\
=\int_{Q\times Y_{1}}f\varphi _{1}+\int_{Q\times Y_{2}}f_{b}\varphi
_{2}+\int_{\Omega \times Y_{1}}h\varphi _{1}+\int_{\Omega \times
Y_{2}}h_{b}\varphi _{2}  \label{fv3}
\end{gather}%
Now, integration by parts in (\ref{fv3}) yields
\begin{gather}
\left\vert Y_{1}\right\vert \partial _{t}T-\alpha \text{\textrm{div}}\left(
\int_{Y_{1}}\left( \nabla T+\nabla _{y}T_{1}\right) \right) +\int_{\Sigma
}\gamma \left( T-T_{b}\right) =\left\vert Y_{1}\right\vert f\text{ in }Q%
\text{;}  \notag \\
\   \label{53} \\
\partial _{t}T_{b}-\alpha _{b}\Delta _{y}T_{b}=\left\vert Y_{2}\right\vert
f_{b}\text{ in }Q\times Y_{2}\text{;}  \label{54} \\
-\alpha \text{\textrm{div}}_{y}\left( \nabla T+\nabla _{y}T_{1}\right) =0%
\text{ in }Q\times Y_{2}\text{;}  \label{55} \\
\alpha \left( \nabla T+\nabla _{y}T_{1}\right) \cdot \nu =0\text{ on }%
Q\times \Sigma \text{;}  \label{56} \\
\alpha _{b}\nabla _{y}T_{b}\cdot \nu =\gamma \left( T-T_{b}\right) \text{ on
}Q\times \Sigma \text{;}  \label{57} \\
y\longmapsto T_{1}\text{ }Y-\text{periodic;}  \label{58} \\
y\longmapsto T_{b}\text{ }Y-\text{periodic;}  \label{59} \\
T\left( 0\right) =\left\vert Y_{1}\right\vert h\text{ in }Q\text{;}
\label{510} \\
T\left( 0\right) =\chi _{2}h_{b}\text{ in }Q\times Y_{2}\text{.}  \label{511}
\\
T=0\text{ on }\Gamma  \label{512}
\end{gather}%
Finally we observe that the equations of the system (\ref{53})-(\ref{512})
are exactly and respectively (\ref{5a}), (\ref{4n}), (\ref{fv4}), (\ref{fv5}%
), (\ref{hs2}), (\ref{fv6}), (\ref{4u}), (\ref{4s}), (\ref{4t}) and (\ref{b1}%
). Therefore we have recovered the same process done in the previous section
and consequently the formal asymptotic expansion method used to construct
the homogenized problem (\ref{h0})-(\ref{h2}) is justified. Thus Theorem \ref%
{dp4t2} is proved.

\end{document}